\documentclass[11pt,reqno]{amsart}
\usepackage[T1]{fontenc}
\usepackage{lmodern}
\usepackage{amsmath,amssymb,amsthm,mathtools}
\usepackage[letterpaper,textwidth=6.25in,textheight=9in,centering]{geometry}
\usepackage{microtype}
\usepackage[hidelinks]{hyperref}
\hypersetup{
  pdftitle={Indiscernible extraction beyond stability},
  pdfsubject={Conditional type bounds, ZFC extraction, and weak compactness}
}
\newtheorem{theorem}{Theorem}[section]
\newtheorem{lemma}[theorem]{Lemma}
\newtheorem{corollary}[theorem]{Corollary}
\newtheorem{fact}[theorem]{Fact}
\theoremstyle{definition}
\newtheorem{definition}[theorem]{Definition}

\DeclareMathOperator{\tp}{tp}
\DeclareMathOperator{\cf}{cf}
\DeclareMathOperator{\Aut}{Aut}

\newcommand{\res}{\mathbin{\upharpoonright}}
\newcommand{\concat}{\mathbin{{}^\frown}}
\newcommand{\bdd}{\operatorname{bdd}^{\mathrm u}}
\newcommand{\C}{\mathfrak C}
\newcommand{\EM}{\mathrm{EM}}

\def\forkindep{\mathrel{\raise0.2ex\hbox{\ooalign{\hidewidth$\vert$\hidewidth\cr\raise-0.9ex\hbox{$\smile$}}}}}
\newcommand{\indbu}{\forkindep^{\mathrm{bu}}}

\title{Indiscernible extraction for theories with bounded $k$-splitting}
\author[Connor Watson]{Connor Watson\textsuperscript{*}}
\thanks{\textsuperscript{*}\url{https://www.cjwsets.com},
email: \texttt{cjwsets@gmail.com}}
\date{}

\begin{document}
\begin{abstract}
Hanson introduced the class of theories with bounded $k$-splitting, a higher-arity notion of stability, and proved an indiscernible extraction result for those theories assuming the existence of a $k$-ineffable cardinal. In this article, we show that the conclusion of Hanson's theorem is equivalent to the weaker assertion that there exists a weakly compact cardinal, as well as prove additional indiscernible extraction results for these theories in ZFC. Finally, we show that Hanson's construction of total bounded ultraimaginary Morley sequences in theories with bounded $k$-splitting follows from ZFC.
\end{abstract}
\maketitle
\section*{Introduction}

There have been several efforts in model theory to introduce higher-arity generalizations of stability, the most well-known classes of these theories being the class of $\mathrm{NIP}_k$ theories introduced by Shelah in \cite{ShelahNIPk} and the class of $\mathrm{(N)FOP}_k$ theories introduced by Terry and Wolf in \cite{TW}. A ``higher-arity generalization of stability'' should be a natural generalization of both stability and $k$-arity.\footnote{A theory $T$ is \textit{k-ary} if every formula is equivalent modulo $T$ to a Boolean combination of formulas with at most $k$ variables.} Recently, in \cite{HansonBoundedksplitting}, Hanson introduced a new higher-arity notion of stability known as \textit{bounded k-splitting}. We will review the definition and give some motivation for the concept in Section 1, but for a more thorough introduction to the notion, as well as its relationship to similar notions already introduced, we refer the reader to \cite{HansonBoundedksplitting}.

In \cite{HansonBoundedksplitting}, assuming the existence of a \textit{k-ineffable cardinal}, Hanson proves an indiscernible extraction theorem for theories with bounded $k$-splitting just as Shelah did for stable theories in \cite{ShelahClassification}. He asks whether this large cardinal assumption can be improved, and whether a potentially weaker indiscernible extraction result can be proven in ZFC. Here, we answer both questions in the affirmative. Furthermore, Hanson gives an application by showing that, under the same large cardinal property, theories with bounded $k$-splitting have certain kinds of \textit{bu-Morley sequences}. We also show that the large cardinal assumptions here are not needed, either.

More specifically, say that the partition relation $\kappa\to(\lambda)_{T,\gamma,B}$ holds if for any sequence of $\gamma$-tuples of length $\kappa$ in a model of $T$, there exists a subsequence of length $\lambda$ which is indiscernible over $B$. The following is shown in \cite{HansonBoundedksplitting}:

\begin{theorem}[Hanson] If $T$ has bounded $k$-splitting, $\gamma$ is any ordinal, and $\lambda>|T|+|B|+|\gamma|$ is a $k$-ineffable cardinal, then $\lambda\to(\lambda)_{T,\gamma,B}.$
\end{theorem}

Here, we show that this large cardinal hypothesis can be weakened, and also that two weaker indiscernible extraction results can be proven just in ZFC. In particular, we show the following:

\begin{theorem}
    Suppose $T$ has bounded $k$-splitting, $B$ is a parameter set of size at most $\lambda,$ and $\gamma$ is an ordinal. Then,

    \begin{enumerate}
        \item If $\lambda>|T|+|B|+|\gamma|$ is weakly compact, then $\lambda\to(\lambda)_{T,\gamma,B}$.
        \item If $\lambda\geq|T|+|B|+|\gamma|$, then $\beth_k(\lambda)^+\to(\lambda^+)_{T,\gamma,B}.$
        \item If $\lambda^{|T|}=\lambda$ and $1\leq n < \omega$, then $\beth_{k-1}(\lambda)^+\to(\lambda^+ + (k-1))_{T,n,B}.$
    \end{enumerate}

    Moreover, the hypothesis that $\lambda$ is weakly compact in (1) is optimal.
\end{theorem}

\noindent\textbf{AI statement.}\quad
No generative AI was used in the conception, generation, or writing of this paper.

\section{Preliminaries}

Throughout, $T$ denotes a complete infinite theory, and we work in a
sufficiently saturated monster $\C$.
For a sequence $(a_i)_{i<\eta}$, let $a_{<\alpha}$ denote the set of
coordinates occurring before $\alpha$, and let $a_I$ denote concatenation
in increasing index order for finite $I\subseteq\eta$.
The sequence is \emph{indiscernible over $B$} if
$a_I\equiv_Ba_J$ whenever $I,J\subseteq\eta$ are finite and have the
same size.

Shelah's original development of stability theory centers heavily around the notion of \textit{splitting}. Recall that a type $\mathrm{tp}(a/B)\in S_n(B)$ \textit{splits over} $C\subseteq B$ if there is a formula $\varphi(x,y)$ and $b,b'\in A$ with $b\equiv_C b'$ such that both $\varphi(a,b)$ and $\neg\varphi(a,b')$ are true. It is shown in \cite{ShelahClassification} that stability can be defined in terms of splitting:

\begin{fact}[Shelah]
$T$ is stable if and only if there is a cardinal $\lambda$ such that for any $a$ and $B$, there is a $C\subseteq B$ with $|C|\leq\lambda$ such that $\mathrm{tp}(a/B)$ does not split over $C$.   
\end{fact}

It is this formulation of stability that motivates the relevant higher-arity notion defined below. As usual, given a list $a_0,...,a_{k-1}$, we write $a_0,...,\hat{a}_i,...,a_{k-1}$ to mean the same list with $a_i$ removed.

\begin{definition}
A \emph{$k$-partitioned type} is a type with its variables partitioned into $k$ sets, written $\tp(a_0;\ldots;a_{k-1}/A)$.
For $C\subseteq A$, it \emph{$k$-splits over $C$} if some finite
same-length tuples $b,b'$ from $A$ satisfy
$$
 b\equiv_{Ca_0\cdots\widehat{a_i}\cdots a_{k-1}}b'\quad(i<k),
 \qquad b\not\equiv_{Ca_0\cdots a_{k-1}}b'.
$$
A theory $T$ has \emph{$\lambda$-bounded $k$-splitting} if every such type
with finite blocks has some $C\subseteq A$, $|C|<\lambda$, over which
it does not $k$-split. $T$ has \emph{bounded $k$-splitting} it has $\lambda$-bounded $k$-splitting for some $\lambda$.
\end{definition}

It is easy to see that stability is equivalent to having bounded $1$-splitting, that $k$-arity implies bounded $k$-splitting, and that bounded $k$-splitting implies bounded $(k+1)$-splitting. Hence, this notion is actually a generalization of $k$-arity and stability. 

In \cite{ShelahClassification}, it is shown that $\lambda^+\longrightarrow(\lambda^+)_{T,n,B}$ at every $\lambda$
for which $T$ is $\lambda$-stable. The aim of this paper is to prove analogous results for theories with bounded $k$-splitting for $k>1.$

The following notion in the $r=1$ case is implicit in \cite{ShelahClassification}.

\begin{definition}
For $1\leq r<\omega$, a sequence $(a_i)_{i<\eta}$ is
\emph{$r$-end-homogeneous over $B$} if
\begin{equation}\label{eq:endhom}
 a_I\equiv_{Ba_{<\alpha}}a_J
 \qquad\bigl(\alpha<\eta,\ I,J\in[\eta\setminus\alpha]^r\bigr).
\end{equation}

$(a_i)_{i<\eta}$ is \emph{end-homogeneous over $B$} if it is
$1$-end-homogeneous over $B$.
\end{definition}

The importance of this condition is that locating an indiscernible sequence amounts to non-splitting and end-homogeneity.

\begin{fact}[{\cite{ShelahClassification}}]\label{lem:ordinary}
If $(a_i)_{i<\eta}$ is end-homogeneous over $B$ and every
$\tp(a_i/Ba_{<i})$ does not split over $B$, then the sequence is
indiscernible over $B$.
\end{fact}

\section{ZFC Results on indiscernibles}

We now prove two indiscernible extraction results for theories with bounded $k$-splitting in ZFC. Our first goal is the following theorem:

\begin{theorem}\label{thm:basic}
If $T$ has bounded $k$-splitting and $\lambda\geq{|T|}+|B|+|\gamma|$, then every sequence of
$\gamma$-tuples of length $\beth_k(\lambda)^+$ has a
$B$-indiscernible subsequence of order type $\lambda^+$.
\end{theorem}

We begin by citing a fact from \cite{HansonBoundedksplitting} on types not splitting in theories with bounded $k$-splitting that will be useful moving forward.

\begin{fact}[{\cite{HansonBoundedksplitting}}]\label{fact:smallbase}
If $T$ has bounded $k$-splitting, then every $k$-partitioned type
whose blocks have length $\gamma$ does not $k$-split over some subset
of its parameter domain of size at most ${|T|}+|\gamma|$.
\end{fact}

As mentioned in the preliminary section, the notion of $r$-end-homogeneous sequences is central to indiscernible extraction in both stable theories and theories with bounded $k$-splitting. The following lemma extracts $r$-end-homogeneous sequences in theories with bounded $k$-splitting.

\begin{lemma}\label{lem:ramification}
Let $1\leq r<\omega$ and let
$\lambda\geq{|T|}+|B|+|\gamma|$ be infinite.
Every sequence of $\gamma$-tuples of length $\beth_r(\lambda)^+$
has an $r$-end-homogeneous subsequence over $B$ of order type
$\lambda^++r$.
\end{lemma}

\begin{proof}
For $r=1$, put $\kappa=(2^\lambda)^+$ and write the given sequence as
$(a_i)_{i<\kappa}$. We recursively construct partitions $\mathcal P_\alpha$, for
$\alpha\leq\lambda^+$, of the sets
$$
S_\alpha=\kappa\setminus
\bigcup_{\beta<\alpha}\{\min(X):X\in\mathcal P_\beta\}.
$$

Let $\mathcal P_0$ partition $\kappa$ by
$i\sim j$ if and only if $a_i\equiv_Ba_j$.
For $X\in\mathcal P_\alpha$, let $D_X$ consist
of $B$ and the coordinates of $a_{\min(X_\xi)}$ for $\xi<\alpha$,
where $X_\xi\in\mathcal P_\xi$ is the unique set containing $X$.
For $\alpha<\lambda^+$, form $\mathcal P_{\alpha+1}$ by partitioning
$X\setminus\{\min(X)\}$, for each $X\in\mathcal P_\alpha$, under
$i\sim_Xj\Longleftrightarrow a_i\equiv_{D_Xa_{\min(X)}}a_j$.
At a nonzero limit $\alpha$, take all nonempty intersections
$\bigcap_{\xi<\alpha}X_\xi$ of decreasing chains with
$X_\xi\in\mathcal P_\xi$.

Inductively, $\mathcal P_\alpha$ partitions the indices not selected as
an earlier $\min(X)$, and $a_i\equiv_{D_X}a_j$ for $i,j\in X$.
For $\alpha<\lambda^+$, $|\mathcal P_\alpha|\leq2^\lambda$: at stage
zero and at successors the relevant parameter sets and the variable
blocks have size at most $\lambda$, and at limits there are at most
$(2^\lambda)^{|\alpha|}=2^\lambda$ chains. Thus at most
$\lambda^+\cdot2^\lambda=2^\lambda$ indices are selected before stage
$\lambda^+$, so $\mathcal P_{\lambda^+}\ne\varnothing$.
Choose $X_{\lambda^+}\in\mathcal P_{\lambda^+}$ and its containing sets
$X_\alpha\in\mathcal P_\alpha$ for $\alpha<\lambda^+$.
Their minima $i_\alpha=\min(X_\alpha)$, $\alpha\leq\lambda^+$, strictly
increase. Since $i_\xi,i_\zeta\in X_\alpha$ for
$\alpha\leq\xi,\zeta\leq\lambda^+$, equality of types over $D_{X_\alpha}$
shows that $(a_{i_\alpha})_{\alpha\leq\lambda^+}$ is end-homogeneous over
$B$.

Now suppose $r > 1$ and assume the lemma holds for all $i<r$. Put $\mu = \beth_{r-1}(\lambda)$, and put
$\rho = \lambda^+ + (r-1)$. The case $r=1$ with $\mu$ in place of $\lambda$ gives an end-homogeneous subsequence over $B$ of the form  $(b_i)_{i<\mu^+} \;^\frown (d)$.
Applying the induction hypothesis at $r-1$ with the original $\lambda$ and parameter set $Bd$, to $(b_i)_{i<\mu^+}$, gives an $(r-1)$-end-homogeneous subsequence $(e_\xi)_{\xi<\rho}$ over $Bd$. Write $e_\xi=b_{j_{\xi}}$, where the indices $j_\xi$ increase, and set $e_\rho=d$. 

It remains to show the sequence $(e_\xi)_{\xi \leq \rho}$
is $r$-end-homogeneous over $B$. Fix $\alpha\leq\rho$ and $I,J\in[(\rho+1)\setminus\alpha]^r.$ We must prove that $e_I\equiv_{Be_{<\alpha}} e_J.$ Put $I^-=I\setminus\{\mathrm{max}(I)\}, J^-=J\setminus\{\mathrm{max}(J)\}$, and $D=Be_{<\alpha}$.

First, replace the last tuple indexed by $I$ with $d$. If $t=\mathrm{max}(I)<\rho$, end-homogeneity of $(b_i)_{<\mu^+}\;^\frown(d)$ gives $$e_t=b_{j_t}\equiv_{Bb_{<j_t}}d,$$ and hence $e_I\equiv_D e_{I^- d}.$ If $\mathrm{max}(I)=\rho$, then this equality holds because $e_\rho=d$. The same argument also gives $e_J\equiv_D e_{J^- d}.$

Now apply $(r-1)$-end-homogeneity over $Bd$. Both $I^-$ and $J^-$ are $(r-1)$-element subsets of $\rho\setminus\alpha$, thus $e_{I^-}d\equiv_{Bde_{<\alpha}} e_{J^-}.$ Since $Bde_{<\alpha}=Dd$, we also have $e_{I^-}d\equiv_D e_{J^-}d,$ so the sequence $(e_\xi)_{\xi\leq\rho}$ is $r$-end-homogeneous over $B.$
\end{proof}

\begin{lemma}\label{lem:tail}
Suppose $T$ has bounded $k$-splitting and
$\lambda\geq{|T|}+|B|+|\gamma|$ is infinite.
If $\cf(\theta)>\lambda$, every $k$-end-homogeneous sequence of
$\gamma$-tuples of length $\theta$ over $B$ has a
$B$-indiscernible final segment.
\end{lemma}

\begin{proof}
Let $(a_i)_{i<\theta}$ be the sequence and put $A_\alpha=Ba_{<\alpha}$.
For each $\alpha<\theta$, define
$$
 p_\alpha(x_0;\ldots;x_{k-1})=\tp(a_I/A_\alpha),
 \qquad I\in[\theta\setminus\alpha]^k.
$$
End-homogeneity makes $p_\alpha$ independent of the choice of $I$
and gives $p_\alpha\subseteq p_\beta$ for $\alpha<\beta$.
Thus $p=\bigcup_{\alpha<\theta}p_\alpha$ is a complete consistent
partitioned type over $Ba_{<\theta}$.
By Fact~\ref{fact:smallbase}, choose $C\subseteq Ba_{<\theta}$ of size
at most ${|T|}+|\gamma|\leq\lambda$ over which $p$ does not
$k$-split. Choose $\beta<\theta$ with $C\subseteq A_\beta$, using
$\cf(\theta)>\lambda$. For every $I\in[\theta\setminus\beta]^k$,
$$
 \tp(a_I/A_{\min I})=p\res A_{\min I}
$$
does not $k$-split over $A_\beta$, which is also true of the realization.

Enumerate the final segment $(a_i)_{\beta\leq i<\theta}$ in order as
$(e_i)_{i<\eta}$, and put $D=A_\beta$. It is
$k$-end-homogeneous over $D$, hence $k$-indiscernible over $D$.
Suppose it is $m$-indiscernible, where $m\geq k$. Write any
increasing $(m+1)$-tuple of its blocks as $uv$, with
$v=(v_0,\ldots,v_{k-1})$ its last $k$ blocks, and put
$u_0=e_0\cdots e_{m-k}$. For each $j<k$, $m$-indiscernibility gives
$$
 u\equiv_{Dv_0\cdots\widehat v_j\cdots v_{k-1}}u_0.
$$
Non-$k$-splitting of the type
of $v$ over its preceding segment gives $uv\equiv_Du_0v$, and
$k$-end-homogeneity gives $u_0v\equiv_De_0\cdots e_m$.
For infinite blocks, each formula uses only finite coordinate subtuples
of $u,u_0$. This proves $(m+1)$-indiscernibility. The final segment is
therefore $D$-indiscernible, hence $B$-indiscernible.
\end{proof}

Theorem \ref{thm:basic} now follows.

\begin{proof}[Proof of Theorem \ref{thm:basic}]
Apply Lemma~\ref{lem:ramification} with $r=k$ and take the first
$\lambda^+$ tuples of the resulting subsequence. This is a
$k$-end-homogeneous sequence of length $\lambda^+$.
Lemma~\ref{lem:tail} applies since $\cf(\lambda^+)=\lambda^+>\lambda$.
A final segment of $\lambda^+$ has order type $\lambda^+$.
\end{proof}

We now know that given a sequence of (possibly infinite) tuples of length $\beth_{k}(\lambda)^+$, we may extract an indiscernible subsequence of length $\lambda^+$. Our next goal is to show that we may improve this result assuming that all of the tuples are finite (though with a slightly stronger assumption on the cardinal $\lambda$). Theorem \ref{thm:main} is almost precisely the classical stability theory result in the $k=1$ case.

\begin{theorem}\label{thm:main}
Suppose $T$ has bounded $k$-splitting, $\lambda^{|T|}=\lambda$, $|B|\leq\lambda$, and $n\geq 1$ is finite. Then every sequence of $n$-tuples of length $\beth_{k-1}(\lambda)^+$ has a $B$-indiscernible subsequence of length $\lambda^+ + (k-1).$
\end{theorem}

This is an improvement of Theorem \ref{thm:basic} in two ways, first by removing a beth from the length of the original sequence, and second by extracting an indiscernible sequence with an extra finite tail. 

The following lemma is an analogue of a classical result of Shelah \cite{ShelahStabletheories} bounding the number
of complete types extending a given type in a stable theory, to the
context of bounded $k$-splitting. 

\begin{lemma}\label{lem:count}
Suppose $T$ has bounded $k$-splitting. Fix a parameter set $A$ and
complete types
$$
p_i(x_0,\ldots,\widehat{x_i},\ldots,x_{k-1})\qquad(i<k)
$$
over $A$, where the variable blocks are finite.
The number of complete types in $x_0,\ldots,x_{k-1}$ over $A$ extending
$\bigcup_{i<k}p_i$ is at most
$$
(|A|+{|T|})^{{|T|}}\cdot 2^{2^{{|T|}}}.
$$
This number is at most $\lambda$ whenever
$|A|\leq\lambda$, $\lambda\geq2^{2^{{|T|}}}$, and
$\lambda^{{|T|}}=\lambda$.
\end{lemma}

\begin{proof}
Let $Q$ be the set of extensions, and assume $Q\ne\varnothing$.
By Fact~\ref{fact:smallbase}, every $q\in Q$ does not $k$-split over some
$C\subseteq A$ of size at most ${|T|}$. There are at most
$(|A|+{|T|})^{{|T|}}$ possible choices for $C$.

Fix $C$ and a realization $a_0,\ldots,a_{k-1}$ of a member of $Q$.
For each $m<\omega$, define an equivalence relation $\sim_C$ on
$A^m$ by setting
$$
\bar b\sim_C\bar b'
\quad\Longleftrightarrow\quad
\bar b\equiv_{C a_0\cdots\widehat{a_i}\cdots a_{k-1}}\bar b'
\quad(i<k).
$$

There are at most $2^{|T|}$ classes for each $m$.
For each $m$, choose one tuple from each equivalence class in $A^m$,
and let $D_C$ consist of $C$ and all coordinates of the chosen tuples.
Then $|D_C|\leq2^{{|T|}}$.

If $q\in Q$ does not $k$-split over $C$, and $\bar b'$ represents
$\bar b$, non-$k$-splitting gives
$$
\varphi(x_0,\ldots,x_{k-1};\bar b)\in q
\quad\Longleftrightarrow\quad
\varphi(x_0,\ldots,x_{k-1};\bar b')\in q.
$$
Thus $q$ is determined by $q\res D_C$. There are at most
$2^{{|T|}+|D_C|}\leq2^{2^{{|T|}}}$ such restrictions, and so summing over $C$
proves the bound.
\end{proof}

\begin{lemma}\label{lem:fodor}
Suppose $T$ has bounded $k$-splitting,
$\lambda\geq{|T|}$,  $\lambda^{|T|}=\lambda$, and $|B|\leq\lambda$.
Assume that every family of types as in Lemma~\ref{lem:count} over any set of
size at most $\lambda$ has at most $\lambda$ common complete extensions.
Then every sequence of $n$-tuples of length $\lambda^+$ that is
$(k-1)$-end-homogeneous over $B$ has a
$B$-indiscernible subsequence of order type $\lambda^+$.
\end{lemma}

\begin{proof}
Let $(a_i)_{i<\theta}$ be the given sequence, where $\theta=\lambda^+$.
Put $A_\alpha=B\cup a_{<\alpha}$ and $A=B\cup a_{<\theta}$.
Adjoin tuples $c_1,\ldots,c_{k-1}$ in an elementary
extension so that $(a_i)_{i<\theta}\concat(c_1,\ldots,c_{k-1})$
is still $(k-1)$-end-homogeneous over $B$.
To do this, realize the increasing union
$$
\bigcup_{\alpha<\theta}\tp(a_I/A_\alpha),
\qquad I\in[\theta\setminus\alpha]^{k-1}.
$$

Put $c=c_1\cdots c_{k-1}$.
Then we have, for every $\alpha<\theta$ and $1\leq j<k$,
\begin{equation}\label{eq:restrictions}
 a_\alpha c_1\cdots\widehat{c_j}\cdots c_{k-1}
 \equiv_{A_\alpha}c_1\cdots c_{k-1}.
\end{equation}

By Fact~\ref{fact:smallbase}, choose $C_\alpha\subseteq A_\alpha$,
$|C_\alpha|\leq{|T|}$, such that
$\tp(a_\alpha;c_1;\ldots;c_{k-1}/A_\alpha)$ does not $k$-split over
$C_\alpha$. We claim that
\begin{equation}\label{eq:ordinary}
 \tp(a_\alpha/A_\alpha c)\text{ does not split over }C_\alpha c.
\end{equation}
Indeed, if $\bar b,\bar b'\in A_\alpha^{<\omega}$ satisfy
$\bar b\equiv_{C_\alpha c}\bar b'$, then \eqref{eq:restrictions} transfers
this equality to the parameters obtained by omitting any one block from
$a_\alpha,c_1,\ldots,c_{k-1}$. Non-$k$-splitting therefore gives
$\bar b\equiv_{C_\alpha a_\alpha c}\bar b'$.
For tuples from $A_\alpha c$, equality over $C_\alpha c$ forces
coordinates belonging to $c$ to match. Delete those coordinates,
apply the preceding argument, and reinsert them.

On the stationary set
$\{\alpha<\theta:\cf(\alpha)={|T|}^+\}$, the coordinates of
$C_\alpha\setminus B$ occur below some index $\beta_\alpha<\alpha$.
Here ${|T|}^+<\theta$, since $2^{{|T|}}\leq\lambda^{{|T|}}=\lambda$.
Fodor's lemma gives a stationary set on which $\beta_\alpha$ is constant,
say with value $\beta$. Since $|A_\beta|\leq\lambda$ and
$\lambda^{{|T|}}=\lambda$, we may thin further to a stationary set
$S$ on which $C_\alpha$ is fixed, say with value $C$.

For each $m<\omega$, let $R_m\subseteq A^m$ contain exactly one
representative of each equivalence class for the equivalence relation
defined by letting
$\bar b\sim_m\bar b'\Longleftrightarrow\bar b\equiv_{Cc}\bar b'$.
Let $D$ consist of $C$ and all coordinates of tuples in
$\bigcup_{m<\omega}R_m$. Since $|Cc|\leq{|T|}$, we have
$|R_m|\leq2^{{|T|}}$ and hence $|D|\leq2^{{|T|}}\leq\lambda$.
Choose $\delta<\theta$ with $D\subseteq A_\delta$, and discard indices
at most $\delta$ from $S$.
For $\alpha\in S$, put
$q_\alpha=\tp(a_\alpha,c_1,\ldots,c_{k-1}/A_\delta)$.
Let $p_i(x_0,\ldots,\widehat{x_i},\ldots,x_{k-1})$ be
$\tp(c/A_\delta)$ with its variable blocks renamed in order $(i<k)$.
By \eqref{eq:restrictions}, $q_\alpha\supseteq\bigcup_{i<k}p_i$.
The hypothesis on common complete extensions, applied to these types
over $A_\delta$, where $|A_\delta|\leq\lambda$, gives at most
$\lambda$ possible $q_\alpha$. Thin to a stationary $S'\subseteq S$
on which they agree.
Thus the $a_\alpha$, $\alpha\in S'$, have the same type over
$A_\delta c$.

If $\alpha<\alpha'$ belong to $S'$ and $\bar b\in A_\alpha^m$,
choose $\bar d\in R_m\subseteq D^m$ with
$\bar b\equiv_{Cc}\bar d$. Applying \eqref{eq:ordinary} at both indices
gives, for every formula $\varphi$,
$$
\varphi(a_\alpha;\bar b,c)
\ \iff\ \varphi(a_\alpha;\bar d,c)
\ \iff\ \varphi(a_{\alpha'};\bar d,c)
\ \iff\ \varphi(a_{\alpha'};\bar b,c).
$$
Consequently $a_\alpha\equiv_{A_\alpha c}a_{\alpha'}$.
The sequence $(a_\alpha)_{\alpha\in S'}$ is end-homogeneous over
$A_\delta c$. By \eqref{eq:ordinary}, restriction and enlargement of the
base show that each successive type does not split over $A_\delta c$.
Lemma~\ref{lem:ordinary} therefore gives indiscernibility over
$A_\delta c$, hence over $B$.
\end{proof}

Theorem \ref{thm:basic} now follows directly from Lemma \ref{lem:fodor}.

\begin{proof}[Proof of Theorem \ref{thm:basic}]
By Lemma~\ref{lem:ramification}, the given sequence
$(a_i)_{i<\beth_{k-1}(\lambda)^+}$ has a $(k-1)$-end-homogeneous
subsequence $(b_j)_{j<\lambda^++(k-1)}$ over $B$, where
$b_j=a_{\xi_j}$ and the indices $\xi_j$ are strictly increasing.
Lemma~\ref{lem:count} shows that the hypothesis of
Lemma~\ref{lem:fodor} is satisfied, so we may apply it  to the first
$\lambda^+$ tuples gives a $B$-indiscernible subsequence indexed by a
cofinal $Y\subseteq\lambda^+$ of order type $\lambda^+$,
say $(c_i)_{i<\lambda^+}$. Then the sequence
$(c_i)_{i<\lambda^+}\;^\frown
(b_{\lambda^+},\ldots,b_{\lambda^++k-2})$
is a $B$-indiscernible sequence.
\end{proof}

\begin{corollary}\label{cor:larger}
Suppose $T$ has bounded $k$-splitting and
$\mu\geq 2^{{|T|}}$. For every positive finite $n$,
$$
\beth_k(\mu)^+\longrightarrow
\bigl((2^\mu)^++(k-1)\bigr)_{T,n,B}.
$$
The same conclusion holds over any parameter set $B$ with $|B|\leq\mu$.
\end{corollary}

\begin{proof}
Set $\lambda=2^\mu$. Then $\lambda\geq2^{2^{|T|}}$,
$\lambda^{|T|}=\lambda$, and
$\beth_{k-1}(\lambda)=\beth_k(\mu)$. Apply Theorem~\ref{thm:main}.
\end{proof}

\section{Extraction from Weak-Compactness}\label{sec:weak}

It is important to note that the indiscernible extraction results in the previous section are weaker than the ones originally proven by Shelah for stable theories or by Hanson for theories with bounded $k$-splitting. In those contexts, an indiscernible subsequence is able to be extracted of the same length of a given sequence. In \cite{HansonBoundedksplitting}, this result for theories with bounded $k$-splitting is achieved assuming the existence of a $k$-ineffable cardinal. Here, we show that this large cardinal assumption can be weakened to the existence of a weakly compact cardinal, and that this is in fact optimal. Recall that an
uncountable cardinal $\kappa$ is \emph{weakly compact} if and only if it
is strongly inaccessible and has the tree property: every tree of height
$\kappa$ whose levels have size less than $\kappa$ has a branch of length
$\kappa$. Equivalently, $\kappa$ is weakly compact if and only if the partition relation $\kappa\to(\kappa)^2_2$ holds.

\begin{theorem}\label{thm:weak}
Suppose $T$ has bounded $k$-splitting, where $1\leq k<\omega$, and
$\kappa>{|T|}+|B|+|\gamma|$ is weakly compact.
Every sequence of $\gamma$-tuples of length $\kappa$ has a
$B$-indiscernible subsequence of order type $\kappa$.
\end{theorem}

\begin{proof}
Choose an infinite ${|T|}+|B|+|\gamma|\leq\lambda<\kappa$.
We first prove, by induction on $r\geq1$, that every sequence of
$\gamma$-tuples of length $\kappa$ has an $r$-end-homogeneous
subsequence of order type $\kappa$.

For $r=1$, carry out the partition construction in
Lemma~\ref{lem:ramification} on the given index set $\kappa$, at every
stage $\alpha<\kappa$. Its levels satisfy
$$
 |\mathcal P_\alpha|\leq2^{\lambda+|\alpha|}<\kappa.
$$
The strict inequality uses the strong inaccessibility of $\kappa$.
Indeed, the right-hand side bounds the number of types over the parameters
on a branch of length $\alpha$, and at a limit it also bounds the number
of preceding chains. Every $\mathcal P_\alpha$ is nonempty by the
regularity of $\kappa$.
Order the nodes $(\alpha,X)$, $X\in\mathcal P_\alpha$, by
$$
 (\alpha,X)<(\beta,Y)\quad\Longleftrightarrow\quad
 \alpha<\beta\text{ and }Y\subseteq X.
$$
Each node has exactly one predecessor at every previous level.
After adjoining a root, this is a tree to which the tree property
applies. The downward closure of a cofinal branch gives classes
$X_\alpha\in\mathcal P_\alpha$ for all $\alpha<\kappa$.
Their minima strictly increase and index an end-homogeneous
subsequence of length $\kappa$, by the same type equalities as in
Lemma~\ref{lem:ramification}.

For $r>1$, first obtain an end-homogeneous subsequence
$(b_i)_{i<\kappa}$. In an elementary extension choose a
$\gamma$-tuple $d$ realizing
$$
 \bigcup_{\alpha<\kappa}\tp(b_i/Bb_{<\alpha}),
 \qquad \alpha\leq i<\kappa.
$$
The types in this union are well-defined and increasing, so the union
is consistent and complete. Consequently
$(b_i)_{i<\kappa}\concat(d)$ is end-homogeneous over $B$.
Apply induction over $Bd$ to $(b_i)_{i<\kappa}$.
Replacing the last block in each compared $r$-tuple by $d$, exactly
as in Lemma~\ref{lem:ramification}, shows that the selected
$(r-1)$-end-homogeneous sequence over $Bd$ is
$r$-end-homogeneous over $B$.

Take $r=k$. Lemma~\ref{lem:tail} applies at $\theta=\kappa$, since
$\cf(\kappa)=\kappa>\lambda$. Its indiscernible final segment has order
type $\kappa$.
\end{proof}

It is clear from the definition of a weakly compact cardinal that this result is optimal. Let
$T_{\mathrm{rg}}$ be the theory of the random graph with edge relation
$R$. It is binary and
therefore has bounded $k$-splitting for every $k\geq2$.
Suppose $\kappa$ is uncountable and
$\kappa\longrightarrow(\kappa)_{T_{\mathrm{rg}},1}$. Given any coloring
$c:[\kappa]^2\to2$, compactness gives distinct vertices
$(a_\alpha)_{\alpha<\kappa}$ in a model of $T_{\mathrm{rg}}$ such that
$$
 R(a_\alpha,a_\beta)\quad\Longleftrightarrow\quad
 c(\{\alpha,\beta\})=1\qquad(\alpha<\beta<\kappa).
$$
Every finite part of these requirements is realized in the random graph.
An indiscernible subsequence of length $\kappa$ gives a homogeneous set
for $c$. Thus $\kappa\longrightarrow(\kappa)^2_2$, hence $\kappa$ is weakly compact.

\section{Total bu-Morley sequences}

As an application, Hanson showed that indiscernible extraction in theories with bounded $k$-splitting implies the existence of certain total bu-Morley sequences. We show that our weaker indiscernible extraction result in ZFC actually suffices for this goal.
We begin by briefly recalling the relevant definitions. For a more thorough discussion of ultraimaginaries, we refer the reader to \cite{HansonBoundedksplitting} and \cite{HansonBU}.

An \emph{ultraimaginary} is a class $c/E$, where $E$ is an
automorphism-invariant equivalence relation on tuples.
Write $\bdd(X)$ for the class of ultraimaginaries with bounded
$\Aut(\C/X)$-orbit. Equivalently, $c/E\in\bdd(X)$ if its orbit
has size at most $2^{|X|+|c|+{|T|}}$, in a sufficiently large monster
\cite{HansonBU}.
We then define an independence relation by setting
$$
 u\indbu_Bv\quad\Longleftrightarrow\quad
 \bdd(Bu)\cap\bdd(Bv)=\bdd(B).
$$
It is invariant under automorphisms and monotone in $u,v$.

A \emph{bu-Morley sequence} over $B$ is a $B$-indiscernible sequence
$(a_i)_{i<\omega}$ such that $a_i\indbu_Ba_{<i}$ for each $i<\omega$.
It is \emph{total} if, for all set-sized linearly ordered sequences $I,J$,
$$
 I+J\equiv_B^{\EM}(a_i)_{i<\omega}
 \quad\Longrightarrow\quad I\indbu_BJ.
$$
Here $I+J$ is concatenation, and equality of Ehrenfeucht--Mostowski
(EM) types means that the types over $B$ of corresponding finite
increasing tuples agree.

We use the following two results from \cite{HansonBU}.

\begin{fact}\label{fact:chain}
If $(c_i)_{i<\omega}$ is bu-Morley over $B$ and indiscernible over
$Be$, where $e$ is any set-sized tuple, then $c_0\indbu_Be$.
\end{fact}

\begin{fact}\label{fact:totality}
A $B$-indiscernible sequence $(a_i)_{i<\omega}$ is total bu-Morley
over $B$ if and only if there are infinite sequences $I,J$ such that
$I+J\equiv_B^{\EM}(a_i)_{i<\omega}$ and $I\indbu_BJ$.
\end{fact}

The next lemma is a version of Hanson's tree construction used in
\cite{HansonBoundedksplitting}.

\begin{lemma}\label{lem:butree}
For every infinite cardinal $\mu$, parameter set $B$, and set-sized
tuple $a$, there are tuples $(d_\xi)_{\xi<\mu}$, with
$d_\xi\equiv_Ba$, and, for each $\xi<\mu$, a bu-Morley sequence
$(c_i^\xi)_{i<\omega}$ over $B$ such that
\begin{equation}\label{eq:treeinput}
 d_{<\xi}\subseteq\operatorname{rng}(c_0^\xi),
 \qquad
 (c_i^\xi)_{i<\omega}\text{ is indiscernible over }Bd_{>\xi}.
\end{equation}
Here $d_{<\xi}$ and $d_{>\xi}$ denote the sets of coordinates of the
earlier and later tuples, respectively.
\end{lemma}

\begin{proof}
Let $\mathcal T_\mu$ be the tree of finitely supported functions
$f:[\beta,\mu)\to\omega$, where $\beta<\mu$ is zero or a successor
ordinal, ordered by extension. The tree construction in \cite{HansonBU} gives
an $s$-indiscernible, bu-spread-out family
$(t_f)_{f\in\mathcal T_\mu}$ over $B$, with $t_f\equiv_Ba$ for every
$f$.

For each $\xi<\mu$, let $f_\xi$ be identically zero on $[\xi+2,\mu)$,
and put $d_\xi=t_{f_\xi}$. For $i<\omega$, set
$g_i^\xi=f_\xi\cup\{(\xi+1,i)\}$ and let $c_i^\xi$ be the
concatenation of the tuples $t_g$ for $g\in\mathcal T_\mu$ with
$g\supseteq g_i^\xi$.
Use enumerations matched by the map that changes only the value at
$\xi+1$. The domains of $f_\xi$ and $g_i^\xi$ start at successor
ordinals, so both are nodes of $\mathcal T_\mu$.
Then $(c_i^\xi)_{i<\omega}$ is bu-Morley over $B$.

If $\eta<\xi$, then $f_\eta\supseteq g_0^\xi$, so every coordinate of
$d_\eta$ occurs in $c_0^\xi$. If $\eta>\xi$, then $f_\eta$ is a restriction
of all the subtrees rooted at $g_i^\xi$. The $s$-indiscernibility
then$(c_i^\xi)_{i<\omega}$ indiscernible over $Bd_{>\xi}$.
This also holds for the whole subtree tuples, since each formula uses
only finitely many coordinates.
\end{proof}

\begin{theorem}\label{thm:totalbu}
Suppose $T$ has bounded $k$-splitting for some $1\leq k<\omega$.
For every parameter set $B$ and every possibly infinite tuple $a$,
there is a total bu-Morley sequence $(a_i)_{i<\omega}$ over $B$
with $a_0=a$.
\end{theorem}

\begin{proof}
Choose an infinite $\lambda\geq{|T|}+|B|+|a|$ and put
$\mu=\beth_k(\lambda)^+$. Theorem~\ref{thm:basic} gives extraction of
order type $\lambda^+$, hence of order type $\omega+\omega$, over
$B$ for tuples of length $|a|$.

Choose the tuples supplied by Lemma~\ref{lem:butree} and extract a
$B$-indiscernible subsequence $(d_{\xi_j})_{j<\omega+\omega}$, where
the $\xi_j$ strictly increase. Put
$$
 I=(d_{\xi_j})_{j<\omega},\qquad
 \rho=\xi_\omega,\qquad
 J=(d_{\xi_{\omega+1+j}})_{j<\omega}.
$$
Then $I+J$ is $B$-indiscernible. Every coordinate of $I$ occurs
in $c_0^\rho$, and $(c_i^\rho)_{i<\omega}$ is indiscernible over
$BJ$. Fact~\ref{fact:chain}, applied with one tuple enumerating the
whole of $J$, gives $c_0^\rho\indbu_BJ$.
Monotonicity gives $I\indbu_BJ$, and Fact~\ref{fact:totality}
shows that $I$ is total bu-Morley over $B$.
Finally, $d_{\xi_0}\equiv_Ba$: an automorphism fixing $B$ sends
$d_{\xi_0}$ to $a$ and sends $I$ to the required sequence.
\end{proof}

\end{document}